\documentclass[11pt, a4paper]{amsart}
\makeatletter
\def\@tocline#1#2#3#4#5#6#7{\relax
  \ifnum #1>\c@tocdepth 
  \else
    \par \addpenalty\@secpenalty\addvspace{#2}%
    \begingroup \hyphenpenalty\@M
    \@ifempty{#4}{%
      \@tempdima\csname r@tocindent\number#1\endcsname\relax
    }{%
      \@tempdima#4\relax
    }%
    \parindent\z@ \leftskip#3\relax \advance\leftskip\@tempdima\relax
    \rightskip\@pnumwidth plus4em \parfillskip-\@pnumwidth
    #5\leavevmode\hskip-\@tempdima
      \ifcase #1
       \or\or \hskip 3em \or \hskip 2em \else \hskip 3em \fi
      #6\nobreak\relax
    \hfill\hbox to\@pnumwidth{\@tocpagenum{#7}}\par%
    \nobreak
    \endgroup
  \fi}
\makeatother
\usepackage{mathrsfs}
\usepackage{amssymb}
\usepackage{enumitem}
\usepackage{mathabx}
\usepackage[all]{xy}
\theoremstyle{plain}
\newtheorem{prop}{Proposition}
\newtheorem*{prop*}{Proposition}
\newtheorem{thm}[prop]{Theorem}
\newtheorem*{thm*}{Theorem}
\newtheorem{cor}[prop]{Corollary}
\newtheorem{lem}[prop]{Lemma}
\newtheorem{question}[prop]{Question}

\newtheorem*{convention*}{Convention}
\newtheorem{thmintro}{Theorem}

\theoremstyle{definition}
\newtheorem*{defn*}{Definition}
\newtheorem{defn}[prop]{Definition}
\newtheorem{rem}[prop]{Remark}

\newtheorem*{rem*}{Remark}
\newtheorem*{rems*}{Remarks}
\newtheorem*{scholium*}{Scholium}
\newtheorem{example}[prop]{Example}
\newtheorem*{example*}{Example}

\newcommand{\ro}{\varrho}
\newcommand{\fhi}{\varphi}
\newcommand{\teta}{\vartheta}
\newcommand{\se}{\subseteq}
\newcommand{\lra}{\longrightarrow}
\newcommand{\wt}{\widetilde}
\newcommand{\inv}{^{-1}}

\newcommand{\NN}{\mathbf{N}}

\newcommand{\RR}{\mathbf{R}}
\newcommand{\ZZ}{\mathbf{Z}}
\newcommand{\GL}{\mathbf{GL}}

\newcommand{\SL}{\mathbf{SL}}

\newcommand{\gr}{\mathfrak{g}}
\newcommand{\sF}{\mathscr{F}}

\newcommand{\sL}{\mathscr{L}}
\newcommand{\sH}{\mathscr{H}}
\newcommand{\sQ}{\mathscr{Q}}
\newcommand{\sR}{\mathscr{R}}

\newcommand{\sT}{\mathscr{T}}

\newcommand{\cl}{^\mathrm{cl}}
\newcommand{\Ramen}{\mathrm{Ramen}}
\newcommand{\Core}{\mathrm{Core}}
\def\No{N\raise4pt\hbox{\tiny o}\kern+.2em}

\begin{document}
\title[Fixed points for bounded orbits in Hilbert spaces]{Fixed points for bounded orbits\\ in Hilbert spaces}
\author[M. Gheysens and N. Monod]{Maxime Gheysens and Nicolas Monod}
\address{EPFL, 1015 Lausanne, Switzerland}
\thanks{Supported in part by the ERC}
\date{3 August 2015, rev.~10 August}
\begin{abstract}
Consider the following property of a topological group $G$: every continuous affine $G$-action on a Hilbert space with a bounded orbit has a fixed point. We prove that this property characterizes amenability for locally compact $\sigma$-compact groups (e.g.~countable groups).

Along the way, we introduce a ``moderate'' variant of the classical induction of representations and we generalize the Gaboriau--Lyons theorem to prove that any non-amenable locally compact group admits a probabilistic variant of discrete free subgroups. This leads to the ``measure-theoretic solution'' to the von Neumann problem for locally compact groups.

We illustrate the latter result by giving a partial answer to the Dixmier problem for locally compact groups.
\end{abstract}
\maketitle

\tableofcontents

\section{Introduction}
A topological group $G$ is \emph{amenable} if every convex compact $G$-space ${K\neq \varnothing}$ has a fixed point. The precise meaning of this definition is that $K$ is a non-empty convex compact subset of a locally convex topological vector space $V$ and that $G$ has a continuous affine action on $K$. It is equivalent to consider only the case where this is given by a continuous affine (or linear) $G$-representation on $V$ preserving $K$. Moreover, we can assume $V$ and $K$ separable if $G$ is, for instance, locally compact $\sigma$-compact.

It is well-known that any such $K$ is isomorphic (i.e.\ affinely homeomorphic)  to a convex compact subspace of a Hilbert space~\cite[p.~31]{Klee55}. Does it follow that amenability is characterized as a fixed point property for affine actions on Hilbert spaces? after all, preserving a weakly compact set in Hilbert space is equivalent to having a bounded orbit. (The distinction between weak and strong compactness will be further discussed in Section~\ref{sec:rems}.)

The answer is a resounding \emph{no}. First of all, an action on $K$ need not extend to the ambient Hilbert space (see Section~\ref{sec:rems}). Moreover, $G$-actions on $V$  preserving $K$ sometimes have fixed points outside $K$ only, compare e.g.~\cite{Bader-Gelander-MonodINV}.

\emph{In any case, even the statement is wrong!} Indeed, there are non-amenable groups with the fixed point property for any continuous affine action on any reflexive Banach space. This holds for instance for the group of all permutations of an infinite countable set, which is non-amenable (as a discrete group). Indeed, Bergman established the strong uncountable cofinality property for this group~\cite{Bergman06} and the latter implies this fixed point property (see Prop.~1.30 of~\cite{Rosendal13}, whose proof does not use the Polish assumption).

In contrast, we prove that such a characterization does hold for countable groups and more generally locally compact $\sigma$-compact groups:

\begin{thmintro}\label{thm:main}
Let $G$ be a locally compact $\sigma$-compact group.

Suppose that every continuous affine $G$-action on a separable Hilbert space with a bounded orbit has a fixed point. Then $G$ is amenable.
\end{thmintro}

\noindent
(In view of the fixed point definition of amenability, this yields a necessary and sufficient condition and it follows furthermore that a fixed point can be found in the closed convex hull of any bounded orbit.)

\medskip

In this setting, we recall that the fixed point property \emph{without} assuming bounded orbits characterises compact groups by a result of Rosendal~\cite[Thm.~1.4]{Rosendal13}.

\medskip
Our proof takes a curious path: we first construct a very specific example of a group without the fixed point property, and then we pull ourselves by our bootstraps until we reach all non-amenable groups. This process is described below; we would be curious to know if there is a direct proof.

In the absence of a direct proof, the scenic route taken to the conclusion leads us to introduce \emph{moderate induction} and to establish the existence of \emph{tychomorphisms} from free groups to non-amenable locally compact groups, after proving a generalization of the theorem of Gaboriau and Lyons to the locally compact setting. This solves the ``measurable von Neumann problem'' for locally compact groups, see Theorem~\ref{thm:tycho} below.

\begin{rem*}
Our task is thus to construct fixed point free actions on Hilbert spaces that have a bounded orbit. We point out that such actions always have some \emph{unbounded} orbit too. Otherwise, an application of the Banach--Steinhaus principle would show that the linear part of the action is uniformly bounded in operator norm; this would however produce a fixed point, for instance by taking a circumcenter under an invariant uniformly convex norm~\cite[Prop.~2.3]{Bader-Furman-Gelander-Monod}, or using Ryll-Nardzewski.
\end{rem*}

\subsection*{Discrete outline of the proof}
We shall first explain our proof in the special case of countable groups without any topology. Our first step is to obtain some example, \emph{any example at all}, of a group $G$ with a fixed point free action on a Hilbert space with a bounded orbit.

\medskip
Let thus $\mu$ be a probability measure on $G$; this amounts to a non-negative function of sum one. Our Hilbert space is $V=\ell^2(G, \mu) /\RR$, the quotient of $\ell^2(G, \mu)$ by the subspace of constant functions. We endow $V$ with the linear representation induced by the left translation action of $G$ on $\ell^2(G, \mu)$, which indeed preserves the subspace of constants. For this action to be well-defined and to be continuous we need to impose a condition on how $\mu$ behaves under translations. It turns out that such a $\mu$ exists for every countable group; it will be constructed as a negative exponential of suitable length functions on $G$.

\medskip
To turn this linear representation into an affine action, we need a $1$-cocycle $G\to V$. The action is fixed point free and with a bounded orbit if this cocycle is non-trivial in cohomology and bounded. The extension of $G$-representations
$$0 \lra \RR \lra \ell^2(G,\mu) \lra V \lra 0$$
can be analysed by standard cohomological arguments and it suffices to show that there is an $\RR$-valued $2$-cocycle on $G$ which is non-trivial in cohomology and bounded. Such cocycles are known to exists for various groups $G$, for instance (compact hyperbolic) surface groups. Thus we have a first example.

\medskip
In order to produce more examples, we want to show that our $G$-action on $V$ can be ``induced'' to an $H$-action on another Hilbert space $W$ whenever $H$ is a group containing $G$. Classically, $W$ would be a space of maps $H/G \to V$. We shall imitate the first step of our construction by considering $\ell^2$-maps with respect to a suitable probability measure on $H/G$; once again, such a measure will exist as soon as $H$ is countable.

\medskip
At this point, we have constructed a fixed point free action on a separable Hilbert space with a bounded orbit for any countable group containing a surface group. The same statement holds with surface groups replaced by free groups since fixed point properties trivially pass to quotients. We now reach a fundamental obstacle popularized by the \emph{von Neumann problem}: the class of groups containing a free subgroup is still far from the class of non-amenable countable groups.

However, it was proved by Gaboriau--Lyons~\cite{Gaboriau-Lyons} that in an ergodic-theoretical sense, any non-amenable discrete group admits free orbits of free groups (and surface groups) as subrelations. As explained in~\cite[\S~5]{MonodICM}, such measure-theoretical analogues of subgroup embeddings, viewed as ``randembeddings'', are suitable for the induction of representations and of cocycles. Therefore, we can complete the proof of Theorem~\ref{thm:main} \emph{for discrete groups} by generalising the above induction method from subgroups to randembeddings.

\subsection*{About the non-discrete case}
A number of interesting new difficulties appear for locally compact groups. It will be helpful that we can consider separately the Lie case and the totally disconnected case, thanks to a product decomposition result based on structure theory~\cite[Thm.~3.3.3]{Burger-Monod3}.

We shall need to prove a generalization of the Gaboriau--Lyons theorem for locally compact groups. However, merely producing orbit subrelations of free groups is useless here; after all, many \emph{amenable} locally compact groups contain free subgroups. Thus, a discreteness condition will enter the generalized statement. 

The appropriate variant of the notion of randembedding used in the discrete case will be called a \emph{tychomorphism}; it is a one-sided version of measure equivalence couplings for locally compact groups. Building on our generalization of the Gaboriau--Lyons theorem, we shall prove:

\begin{thmintro}\label{thm:tycho}
Let $G$ be a locally compact second countable group.

If $G$ is non-amenable, then there is a tychomorphism from the free group $F_r$ to $G$ for all $0\leq r\leq \aleph_0$.
\end{thmintro}

\noindent
(This yields a necessary and sufficient criterion: in the converse direction, there are several straightforward ways to verify that a group admitting a tychomorphism from $F_r$ for some $2\leq r\leq \aleph_0$ is non-amenable; see for instance~\cite[\S 5]{MonodICM}, where discreteness is not essential.)

\medskip

Once tychomorphisms have been established, we will use them again to perform a non-standard induction of affine actions by means of a \emph{moderate} measure on the locally compact group. In this setting, the existence of moderate measure is more delicate to establish.

\subsection*{An application to the Dixmier problem}

Prompted by the classical unitarisation theorem of Sz\H{o}kefalvi-Nagy~\cite{Szo-Nagy47}, the following problem arose in 1950~\cite{Day50, Dixmier50, Nakamura-Takeda51}: are amenable groups the only \emph{unitarisable} groups, i.e.\ groups for which every uniformly bounded representation on a Hilbert space can be unitarised? We refer to Pisier~\cite{Pisier_survey, PisierLNM} for a thorough exposition and many results.

Since groups containing free subgroups are known not to be unitarisable, it is tempting to appeal to the Gaboriau--Lyons theorem (see Problem~N in~\cite{MonodICM}). This lead to partial answers~\cite{Epstein-Monod, Monod-Ozawa2009}.

\smallskip
There is no reason to restrict the Dixmier problem to discrete groups, and indeed the first examples of non-unitarisable representations were for the Lie group $\SL_2(\RR)$~\cite{Ehrenpreis-Mautner, Kunze-Stein}. Using Theorem~\ref{thm:tycho}, we establish that every non-amenable locally compact group is indeed non-unitarisable \emph{after replacing it} by an extension by an amenable kernel, a statement faithfully parallel to the main result of~\cite{Monod-Ozawa2009} for discrete groups. More precisely, the extension consists in taking a wreath product with a commutative (discrete) group:

\begin{thmintro}\label{thm:wr}
Let $G$ be any locally compact group. For any infinite abelian (discrete) group $A$, the following assertions are equivalent.

\begin{itemize}
\item[(i)] The group $G$ is amenable.

\item[(ii)] The locally compact group $A\wr_{G/O} G$ is unitarisable, where $O<G$ is a suitable open subgroup.\label{pt:wr:wr}
\end{itemize}
\end{thmintro}

There are however several important differences with the discrete case. The first problem is that we do not know if unitarisability passes to (closed) subgroups in the non-discrete case. We do not even know if containing a discrete free subgroup is of any help, which should curb our enthusiasm for tychomorphisms! We can nonetheless prove the above result by combining ergodic methods with structure theory.

A smaller issue is that the category of locally compact groups does not admit full wreath products. This explains the permutational wreath product appearing in Theorem~\ref{thm:wr}. More precisely, $A\wr_{G/O} G$ denotes the topological semi-direct product of  $G$ with the discrete group $\bigoplus_{G/O} A$ which is endowed with a continuous $G$-action since $O$ is open. Thus indeed $A\wr_{G/O} G$ is locally compact, and moreover it is $\sigma$-compact (respectively second countable) when $G$ is so, provided $A$ is countable.


\section{An initial construction}\label{sec:initial}
The goal of this section is to prove that there are \emph{some} groups $G$ with a fixed point free action on a Hilbert space (by continuous affine operators) with a bounded orbit. The construction below applies for instance to the fundamental group of any closed hyperbolic surface. Since the property of having such an action can trivially be pulled back from a quotient group, it follows immediately that it also holds for free groups on at least four generators (the minimal number of generators for such a surface group).

\medskip
Consider a group $G$ (without topology) admitting a non-zero class $\omega$ in degree two cohomology with real coefficients. Assume moreover that $\omega$ can be represented by a \emph{bounded} cocycle; in other words, $\omega$ lies in the image of the comparison map
$$\mathrm{H}^2_\mathrm{b} (G, \RR)  \lra \mathrm{H}^2 (G, \RR)$$
from bounded to ordinary cohomology. This situation arises for instance when $G$ is the fundamental group of a closed hyperbolic surface and $\omega$ is given by the fundamental class of that surface, see e.g.~\cite[\S6]{Thurston_unpublished}.

Furthermore, assume that $G$ admits a probability measure $\mu$ such that the left translation linear $G$-representation on $\ell^2(G, \mu)$ is well-defined and consists of bounded operators. As we shall see in Section~\ref{sec:lengths-measures}, such a measure exists on every countable group $G$ and more generally on many locally compact groups. For the present purposes, it is much easier to justify its existence by assuming that $G$ is a finitely generated (discrete) group, which is the case in the example of surface groups. Indeed, in that case we can choose a word length $\ell$ on $G$, a constant $D>1$ large enough so that the map $g\mapsto D^{-\ell(g)}$ is summable on $G$ and a normalization constant $k>0$ given by the inverse of that sum. Then one checks that $\mu(g) = k  D^{-\ell(g)}$ gives a measure with the desired properties; we refer to Section~\ref{sec:lengths-measures} for a detailed construction in a more general topological case.

\medskip

Let now $V$ (resp.~$E$) be the quotient space of $\ell^2 (G, \mu)$ (resp.~$\ell^\infty (G)$) by the subspace of constant functions. Since $\mu$ is a probability measure, we obtain a commutative diagram
$$\xymatrix{
0 \ar[r] & \RR \ar[d]^{=}\ar[r]  &\ell^\infty (G)\ar[d]\ar[r] & E \ar[d]\ar[r] & 0 \\
0 \ar[r] & \RR \ar[r]  &\ell^2(G,\mu) \ar[r] & V \ar[r] & 0
}$$
where the rows are exact and the vertical arrows are $G$-equivariant, linear, injective and of norm~$\leq 1$.

The idea is to apply the long exact sequence of bounded cohomology~\cite[8.2.1(i)]{Monod} to the first row and the long exact sequence of ordinary cohomology to the second row. (The second row does not behave well for bounded cohomology because it carries representations that are not uniformly bounded.) More precisely, recall that $\mathrm{H}^n_\mathrm{b} (G, \ell^\infty (G))$ vanishes for all $n\geq 1$, see~\cite[4.4.1 and~7.4.1]{Monod}; therefore, by naturality of the comparison map and of the long exact sequences, we have a commutative diagram
$$\xymatrix{
0 \ar[r] & {\mathrm{H}^1_\mathrm{b} (G, E)} \ar[r]\ar[d] & {\mathrm{H}^2_\mathrm{b} (G, \RR)} \ar[r]\ar[d] & 0\\
\cdots \ar[r] &{\mathrm{H}^1 (G, V)} \ar[r] & {\mathrm{H}^2 (G, \RR)} \ar[r] & \cdots
}$$
with exact rows. Now, our assumption that $\omega\neq 0$ lies in the image of ${\mathrm{H}^2_\mathrm{b} (G, \RR)}$ implies that there is a bounded $1$-cocycle $b\colon G\to E$ whose image in ${\mathrm{H}^1 (G, V)}$ is non-trivial. In other words, the corresponding affine $G$-action on $V$ has no fixed point, although the orbit of $0\in V$ under this action is bounded since $b$ remains bounded as a map $G\to E \to V$.

\begin{rem}
A reader wishing to avoid the cohomological language can check the argument by hand as follows. Represent $\omega$ by a bounded map $c\colon G^2 \to \RR$ satisfying the cocycle relation $c(y, z) - c(xy, z) + c(x, yz) - c(x, y)=0$ for all $x,y,z\in G$. Then $b(g)$ is defined as the class modulo~$\RR$ of the function $x\mapsto c(x\inv, g)$ and all properties can be painstakingly verified.
\end{rem}

\section{Moderate lengths and measures}\label{sec:lengths-measures}
It is well-known (and obvious) that the size of a ball in a finitely generated group endowed with a word length grows at most exponentially with the radius. For general countable groups, one can choose another length function to keep this growth control (compare Remark~\ref{rem:length} below). This remains possible more generally in a topological setting, but requires a more delicate analysis which we now undertake.

\begin{defn}\label{def:length}
A \emph{length} on a group $G$ is a function $\ell\colon G\to \RR_+$ such that
\begin{enumerate}[label=(\roman*)]
\item $\ell(g) = \ell(g\inv)$ for all $g\in G$,\label{it:length:sym}
\item $\ell(gh) \leq \ell(g)+\ell(h)$ for all $g,h\in G$.\label{it:length:subadd}
\end{enumerate}
When $G$ is a locally compact group, a length $\ell$ is \emph{moderate} if moreover
\begin{enumerate}[label=(\roman*)]\setcounter{enumi}{2}
\item the ball $B(r)=\ell\inv([0,r])$ is compact for all $r\geq 0$,\label{it:length:cpct}
\item for any Haar measure $m_G$ there is $C\geq 1$ such that $m_G \big(B(r) \big)\leq C^r$ for all $r\geq 1$.\label{it:length:exp}
\end{enumerate}
\end{defn}

Observe that a moderate length is in particular lower semi-continuous thanks to~\ref{it:length:cpct}.

\begin{rem}
We have not specified in~\ref{it:length:exp} whether the Haar measure is left or right, but this is irrelevant in view of the symmetry~\ref{it:length:sym}. Moreover, it suffices to check~\ref{it:length:exp} for one Haar measure since the condition survives scaling upon changing $C$. Likewise, it suffices to check~\ref{it:length:exp} for integer $r$.
\end{rem}

\begin{example}\label{ex:Guivarch}
If $G$ is generated by a \emph{compact} symmetric neighbourhood $U$ of the identity, then the associated word length $\ell(g) = \min\{n\in\NN : g\in U^n \}$ is easily seen to be moderate. (A more refined statement can be found e.g. in Theorem~I.1 of~\cite{Guivarch73}.)
\end{example}

We shall need to go beyond the compactly generated case; notice however that $\sigma$-compactness is a necessary condition in view of~\ref{it:length:cpct}.

\begin{prop}\label{prop:length}
Let $G$ be a totally disconnected locally compact $\sigma$-compact group. Then there exists a continuous moderate length $\ell\colon G\to \NN$.
\end{prop}

Christian Rosendal pointed out to us that a stronger statement was established as one of the main results of the manuscript~\cite{Haagerup-Przybyszewska} (Theorem~5.3). We still provide the proof below since it is shorter and sufficient for our purposes.

\medskip
Before proceeding to the proof, we record the construction of general \emph{Cayley--Abels graphs} originating in~\cite[Beispiel~5.2]{Abels74}, see also~\cite[\S3.4]{Burger-Monod3}. The basic data is a locally compact group $G$, a compact-open subgroup $K<G$ and a subset $S\se G$ satisfying
$$S=S\inv \ \text{ and }\ S=KSK.$$
The associated Cayley--Abels graph $\gr(G,K,S)$ is the graph on the vertex set $G/K$ with edge set $\{(gK, gsK) : g \in G, s \in S\}$. Thus $G$ acts by automorphisms on $\gr(G,K,S)$. This graph is connected if and only if $S$ generates $G$ and it is locally finite if and only if $S$ is compact.

\medskip
For later use in Section~\ref{sec:GL}, we shall extend this construction as follows under the assumptions that $S$ generates $G$, is compact and contains the identity: For every integer $n\geq 1$, let $\gr^n(G,K,S)$ be the (multi-)graph on the vertex set $G/K$ with as many edges between two given vertices as there are paths of length $n$ in $\gr(G,K,S)$ connecting them. Then each $\gr^n(G,K,S)$ remains locally finite, connected and with a vertex-transitive $G$-action.

\begin{proof}[Proof of Proposition~\ref{prop:length}]
 Let $K$ be a compact open subgroup of $G$ and $S$ be any symmetric generating set; upon replacing it by $KSK$, we can moreover assume $S = KSK$. Consider the graph $\gr(G,K,S)$. We will give weights to its edges and then consider the induced path length on the vertex set. There is a natural $S/K$-labeling on the edges; however, it is not invariant under the action of $G$. Indeed, an $sK$-labeled edge and a $tK$-labeled edge are in the same $G$-orbit if and only if $KsK = KtK$, i.e.\ if $sK$ and $tK$ are in the same orbit for the natural action of $K$ on $S/K$. The latter orbit is of size $|K : K \cap s K s\inv|$ and hence finite since $K$ is compact and open. We enumerate these $K$-orbits, recalling that $S/K$ is countable, as $G$ is $\sigma$-compact. The weight is given inductively to each element of the $i$-th orbit as the smallest power of $2$ that is at least as large as the size of the $i$-th orbit and strictly larger than all the weights previously given.

Having attributed a left $K$-invariant weight to each element of $S/K$, we obtain a $G$-invariant weight on the edges of $\gr(G,K,S)$.  Consider now the vertex set $G/K$ with the metric given by the shortest (weighted) path distance. The $G$-action on $G/K$ is isometric for this distance; in particular, the size of a ball of radius $n$ does not depend on its center and we denote by $\beta(n)$ this number, which is finite. Write also $\fhi(k)$ for the number of elements of weight $k$ in $S/K$. Observing that a path of length~$\leq n$ has to start with an edge of weight~$\leq n$, we get a coarse bound:
\begin{equation*}
 \beta(n) \leq \fhi(1) \beta(n - 1) + \fhi(2)\beta(n - 2) + \dots + \fhi(n-1)\beta(1) + \fhi(n).
\end{equation*}
Most of the terms in the latter sum vanish because, by our choice of weights, $\fhi(k)$ is zero whenever $k$ is not a power of $2$. Moreover, since $\beta$ is a non-decreasing function, we have
$$\beta(r) \leq \frac{\beta(r) + \beta(r+1) + \dots + \beta(r+m)}{m+1}$$
for all $r,m$. Putting these two observations together for $r=n-2^j$ and $m=2^{j-1}-1$, we get:
\begin{equation*}
 \beta(n) \leq \fhi(1)\beta(n-1) + \sum_{j = 1}^{\log_2 n} \frac{\fhi(2^j)}{2^{j-1}} \sum^{2^{j-1} + 1}_{p=2^j} \beta(n - p) \leq 2 \sum_{p=0}^{n-1} \beta(p),
\end{equation*}
where the last inequality comes from $\fhi(2^j) \leq 2^j$, which holds by our choice of weights. This estimate implies that $\beta$ grows at most exponentially, indeed that $\beta(n)\leq 2\cdot 3^{n-1} < 3^n$ for all $n\geq 1$.

Getting back to the group $G$, define $\ell_0 (g)$ as the distance from $K$ to $gK$ in $G/K$ for the above weighted distance. So far, $\ell_0$ is a continuous function satisfying~\ref{it:length:subadd} and~\ref{it:length:cpct} for the $\ell_0$-balls $B_{\ell_0}(r)$. Moreover, for any \emph{left} Haar measure, we have $m_G\big(B_{\ell_0} (r)\big) = \beta(r) m_G(K)$, hence the balls grow at most exponentially. We normalize $m_G$ so that $m_G(K)$ is an integer. Finally, set $\ell(g) = \ell_0 (g) + \ell_0 (g^{-1})$. One checks that $\ell$ has all the desired properties.
\end{proof}

\begin{rem}\label{rem:length}
For a discrete countable group $G$, the above argument can be considerably shortened: choose any generating set $S$ such that $S \cap S\inv$ contains only involutions and enumerate $S = \{s_1, s_2, \dots\}$. Then the weighted word length where $s_i$ and $s_i^{-1}$ are given weight $i$ will satisfy properties~\ref{it:length:sym}--\ref{it:length:exp}.

Alternatively, one can simply restrict to $G$ the word length of a finitely generated group containing $G$, which exists by~\cite{HNN}. This overkill, however, cannot be generalized to non-discrete groups because they need not embed into compactly generated groups~\cite{Caprace-Cornulier14}; thus the need for Proposition~\ref{prop:length} remains.
\end{rem}

\begin{defn}\label{def:measure}
A \emph{moderate measure} on a locally compact group $G$ is a probability measure $\mu$ in the same measure class as the Haar measures and such that
\begin{enumerate}[label=(\roman*)]
\item for all $g \in G$, the Radon--Nikod\'ym derivative $\mathrm{d}g \mu/\mathrm{d}\mu$ is essentially bounded on $G$,
\item the map $g \mapsto \left\Vert \mathrm{d}g \mu/\mathrm{d}\mu \right\Vert_\infty$ is locally bounded on $G$.
\end{enumerate}
\end{defn}

The point of this definition is that it readily implies the following:

\smallskip\itshape\noindent
If $\mu$ is a moderate measure on $G$, then the left translation representation of $G$ on $L^2(G,\mu)$ is a well-defined continuous linear representation which is locally bounded (in operator norm).

\smallskip\upshape
To be completely explicit, the (non-unitary) representation above is defined by $(gf)(x) = f(g\inv x)$ for $g,x\in G$ and $f\in L^2(G,\mu)$. In particular, the constant functions constitute a $G$-invariant subspace. The statement above is a particular case of the \emph{moderate induction} that will be investigated in detail in Section~\ref{sec:induction}, to which we refer for a proof.

\medskip
We shall obtain moderate measures thanks to moderate lengths:

\begin{prop}\label{prop:exp}
If $\ell$ is a moderate length on a locally compact group $G$ and $m_G$ a left Haar measure, then the measure $\mu$ defined by
$$\mathrm{d}\mu (x) = k D^{-\ell(x)} \mathrm{d}m_G (x)$$
is moderate when $D\geq 1$ is large enough and $k>0$ is a suitable normalization constant.
\end{prop}

\begin{proof}
Choose any $D>C$, where $C\geq 1$ is as in Definition~\ref{def:length}\ref{it:length:exp}. Since the function $\ell$ is Borel thanks to Definition~\ref{def:length}\ref{it:length:cpct}, it follows that the formula $\mathrm{d}\mu (x) = k D^{-\ell(x)} \mathrm{d}m_G (x)$ makes sense and defines a measure in the same class as $m_G$ for any $k>0$.  In particular, the Radon--Nikod\'ym derivative of Definition~\ref{def:measure} exists. This measure is finite because of $D>C$ and hence it can be normalized by the appropriate choice of $k$. It now suffices to show that for every compact set $U\se G$ the function $\mathrm{d}g \mu/\mathrm{d}\mu (x)$ is bounded uniformly over $g\in U, x\in G$. Since $m_G$ is left invariant, we have $\mathrm{d}g \mu/\mathrm{d}\mu (x) = D^{\ell(x) - \ell(g\inv x)}$, which is bounded above by $D^{\ell(g)}$ in view of Definition~\ref{def:length}\ref{it:length:subadd}. Therefore, it only remains to see that the $\ell$-balls $B(r)$ contain $U$ when $r$ is sufficiently large. Since $U$ is compact and $B(r)$ closed, this a direct application of Baire's theorem using the relation $B(r) B(s) \se B(r+s)$ which follows from Definition~\ref{def:length}\ref{it:length:subadd}.
\end{proof}

\begin{cor}\label{cor:ex:moderate}
Let $G$ be a locally compact group. If $G$ is either compactly generated (e.g.\ connected) or totally disconnected and $\sigma$-compact, then it admits a moderate measure.
\end{cor}

\begin{proof}
In view of Proposition~\ref{prop:exp}, the first case follows from Example~\ref{ex:Guivarch} and the second from Proposition~\ref{prop:length}.
\end{proof}

We record the following for later use.

\begin{lem}\label{lem:measindfin}
 Let $G$ be a locally compact group and $H$ be an open subgroup of finite index in $G$. If $H$ admits a moderate measure, then so does $G$.
\end{lem}

\begin{proof}
Let $\{r_1, \dots, r_n\}$ be a set of representatives in $G$ for the left cosets of $H$. Let $\mu_H$ be a moderate measure on $H$. Define the measures $\mu_i$ on $G$ by $\mu_i = \frac{1}{n} r_i \mu_H$ and let $\mu=\mu_1+\cdots +\mu_n$. As $H$ is open, $\mu$ is a probability measure in the class of the Haar measures. Let $g \in G$. As each $\mu_i$ is absolutely continuous with respect to $\mu$ (actually, bounded by $\mu$), we may write
\begin{equation*}
 \frac{\mathrm{d}g \mu}{\mathrm{d}\mu} = \sum_{i=1}^n \frac{\mathrm{d}g \mu_i}{\mathrm{d}\mu}.
\end{equation*}
For each $i$, denote by $h_{g, i}$ the unique element of $H$ defined by $g r_i = r_j h_{g, i}$ for some $j$ (which also depends on $g$ and $i$). Note that $g \mu_i$ is then equivalent to $\mu_j$. We thus have
\begin{equation*}
 \frac{\mathrm{d}g \mu_i}{\mathrm{d}\mu} = \frac{\mathrm{d}g \mu_i}{\mathrm{d}\mu_j} \cdot  \frac{\mathrm{d} \mu_j}{\mathrm{d}\mu}.
\end{equation*}
The first factor is bounded by $\left\Vert \mathrm{d}h_{g, i} \mu_H/\mathrm{d}\mu_H \right\Vert_\infty$ while the second one is bounded by~$1$. To ensure that $g \mapsto \left\Vert \mathrm{d}g \mu/\mathrm{d}\mu \right\Vert_\infty$ is indeed locally bounded, it is therefore enough to check that for each $i$ the map $g \mapsto h_{g, i}$ sends compact sets of $G$ to relatively compact sets of $H$ since $\mu_H$ is moderate. This holds true because $H$ is open.
\end{proof}

\section{Tychomorphisms}\label{sec:tych}
\begin{flushright}
\begin{minipage}[t]{0.75\linewidth}\itshape\small
Simplicibus itaque verbis gaudet Mathematica Veritas, cum etiam per se simplex sit Veritatis oratio
\begin{flushright}
\upshape\small
Tycho Brahe, \emph{Epistol\ae\ astronomic\ae}, Uraniborg 1596\\
(General preface, p.~23 l.~32 in Dreyer's edition;\\
absent from the 1601 edition of Levinus Hulsius)
\end{flushright}
\end{minipage}
\end{flushright}
\vspace{3mm}
The goal of this section is to discuss basic facts about \emph{tychomorphisms}, a probabilistic variant of closed (e.g.\ discrete) subgroup embeddings in locally compact groups. The closedness condition, which is essential in the context of amenability, is an aspect absent from the analogous concepts of ``orbit subrelation'' and ``randembedding'' of discrete groups. Nevertheless, it will not appear in topological terms, but rather as an ergodic-theoretical smoothness condition following from the definitions below; this reformulation is a special case of the Glimm--Effros dichotomy.

\medskip

Let $G$ be a locally compact group. The Haar measures of $G$ define a canonical measure class on $G$, which is standard if $G$ is second countable. A \emph{measured $G$-space} is a measured space $(\Sigma, m)$ together with a $G$-action such that the action map $G\times \Sigma \to \Sigma$ is non-singular. We shall always consider standard measured spaces, so that all basic tools such as the Fubini--Lebesgue theorem are available. Unless otherwise stated, groups shall act on themselves by left multiplication. We will note by $m_G$ a choice of a (non-zero) left Haar measure on $G$ and by $\widecheck{m}_G$ the corresponding right Haar measure defined as the image of $m_G$ by the inverse map. The modular homomorphism $\Delta_G \colon G \to \RR_+^\ast$ is defined by $\mathrm{d}m_G (xg) = \Delta_G (g) \mathrm{d}m_G (x)$; recall moreover that $\Delta_G \mathrm{d}\widecheck{m}_G  = \mathrm{d}m_G $.

\begin{defn}
Let $(\Sigma, m)$ and $(\Sigma', m')$ be two measured $G$-spaces. We say that $\Sigma$ is an \emph{amplification} of $\Sigma'$ if there is a measure preserving $G$-equivariant isomorphism bewteen $(\Sigma, m)$ and the product of $(\Sigma', m')$ with a measured space $(X,\teta)$ endowed with the trivial $G$-action. The amplification is said to be \emph{finite} if $\teta$ is finite. Remark that if $G$ preserves one of the measures $m$ or $m'$, then it also preserves the other one.
\end{defn}

\begin{example}\label{ex:subgroup}
Let $H$ be a closed subgroup of a locally compact second countable group $G$ and let $m_H, m_G$ be left Haar measures. Then the left $H$-action on $(G, m_G)$ is an amplification of $(H, m_H)$ (see e.g.~\cite{Ripley76}). The latter is finite if $H$ has finite invariant covolume in $G$.
\end{example}

\begin{example}\label{ex:count}
If $G$ is countable (hence discrete), then $(\Sigma, m)$ is an amplification of $G$ if and only if $G$ admits a measurable fundamental domain in $\Sigma$.
\end{example}

\begin{defn}
Let $G$ and $H$ be locally compact second countable groups. A \emph{tychomorphism} from $H$ to $G$ is a measured $G\times H$-space $(\Sigma, m)$ which as a $G$-space is a finite amplification of $(G, \widecheck{m}_G)$ and as an $H$-space is an amplification of $(H, m_H)$.
\end{defn}

Note that $H$ preserves the measure $m$ on $\Sigma$, whereas $G$ does so if and only if it is unimodular. For unimodular groups, the symmetric situation (i.e.\ when $\Sigma$ is also finite as an amplification of $H$) is the usual \emph{measure equivalence} studied e.g.\ in~\cite{Furman_z60}.

\begin{rem}\label{rem:cocycle}
Let $(\Sigma, m)$ be a tychomorphism from $H$ to $G$ and consider isomorphisms $(\Sigma, m) \cong (G, \widecheck{m}_G)\times (X,\teta)$ and $(\Sigma, m) \cong (H, m_H)\times (Z,\eta)$ as in the definition. If we transport the $H$-action through the first isomorphism, we obtain a non-singular $H$-action on $X$ and a measurable cocycle $\alpha\colon H\times X \to G$ such that
$$h (g, s) = (g \alpha(h, s)\inv , hs)$$
because the $H$-action must commute with the $G$-action on $G\times X$. In particular, this $H$-action on $X$ preserves the finite measure $\teta$. On the other hand, if we transport the $G$-action through the second isomorphism, we obtain also a non-singular $G$-action on $Z$ and a measurable cocycle $\beta\colon G \times Z \to H$ such that 
$$g (h, s) = (h \beta(g, s)\inv , gs)$$
but now the $G$-action on $Z$ might not preserve the measure $\eta$. Indeed, one checks that the Radon--Nikod\'ym derivative $\mathrm{d}g\eta/ \mathrm{d}\eta$ is equal to $y \mapsto \Delta_H (\beta(g\inv, y)) \Delta_G (g)$.
\end{rem}

\begin{example}\label{ex:tycho}
If $H$ is a closed subgroup of a locally compact second countable group $G$, then $(G, m_G)$ is a natural tychomorphism from $H$ to $G$. Indeed, it is an amplification of $(H, m_H)$ by Example~\ref{ex:subgroup}; the \emph{right} $G$-action on $(G, m_G)$ commutes with $H$ and is intertwined by the inverse map to $(G, \widecheck{m}_G)$, hence is a finite (trivial) amplification of $(G, \widecheck{m}_G)$.
\end{example}

The next lemma shows how to compose tychomorphisms.

\begin{lem}
\label{lem:comp}
Let $G_1$, $G_2$ and $H$ be locally compact second countable groups. If there is a tychomorphism from $G_1$ to $H$ and another from $H$ to $G_2$, then there is one from $G_1$ to $G_2$.
\end{lem}

\begin{proof}
 Let $(G_1, m_{G_1}) \times X_1 \cong \Sigma_1 \cong (H, \widecheck{m}_H) \times Z_1$ and $(G_2, \widecheck{m}_{G_2}) \times X_2 \cong \Sigma_2 \cong (H, m_H) \times Z_2$  be the tychomorphisms, with $Z_1$ and $X_2$ having finite measure. Consider the corresponding cocycles $\alpha_i\colon G_i \times Z_i \to H$ and the non-singular actions of $G_i$ on $Z_i$ as in Remark~\ref{rem:cocycle}. We endow the space $\Sigma := Z_1 \times Z_2 \times (H, m_H)$ with a non-singular action of $G_1 \times G_2$ by defining
$$(g_1, g_2) (z_1, z_2, h) = (g_1 z_1, g_2 z_2, \alpha_1 (g_1, z_1) h \alpha_2 (g_2, z_2)\inv).$$
The measure-preserving equivariant isomorphisms given by the tychomorphisms show that $\Sigma$ is $G_2$-equivariantly isomorphic to $(G_2, \widecheck{m}_{G_2}) \times Z_1 \times X_2$, hence it is a finite amplification of $(G_2, \widecheck{m}_{G_2})$. Likewise, $\Sigma$ is $G_1$-equivariantly isomorphic to $(G_1, m_{G_1}) \times Z_2 \times X_1$, hence it is an amplification of $G_1$, except that this time we need to precompose the isomorphisms by the map $\Sigma \to Z_1 \times Z_2 \times (H, \widecheck{m}_H)\colon (z_1, z_2, h) \mapsto (z_1, z_2, h\inv)$.
\end{proof}

By the universal property, an embedding of a free group into a quotient group can be lifted; moreover the lift is discrete if the original embedding was so. It turns out that the corresponding fact holds for tychomorphisms.

\begin{prop}\label{prop:liftfree}
Let $\wt{G}$ be a locally compact second countable group, $N$ a closed normal subgroup and $G$ the quotient group $\wt{G} / N$. Let $0\leq r\leq \aleph_0$. If there is a tychomorphism from the free group $F_r$ to $G$, then there is one from $F_r$ to $\wt{G}$.
\end{prop}

\begin{proof}
 Let $S$ be a basis of a free group $F_r$ and $\Sigma \cong G \times X$ be a tychomorphism from $F_r$ to $G$ with associated cocycle $\alpha\colon F_r \times X \to G$. By Example~\ref{ex:count}, there is a fundamental domain $\sF \subset \Sigma$ for the $F_r$-action.

Choose a Borel section $\tau\colon G \to \wt{G}$ of the projection map $\pi\colon \wt{G} \to G$. The following lemma can be checked e.g.\ by applying the cocycle relation to the unique representation of elements of $F_r$ as reduced $S$-words.

\begin{lem}
There is a measurable cocycle $\wt{\alpha}\colon F_r \times X \to G$ such that $\alpha = \pi\circ\wt\alpha$. Moreover, there is an essentially unique such $\wt{\alpha}$ satisfying $\wt{\alpha}(s, x) = \tau (\alpha(s, x))$ for almost all $x \in X$ and all $s \in S$.\qed
\end{lem}

\noindent
(Note that $\wt{\alpha}$ is not in general equal to $\tau \circ \alpha$.)

\smallskip
Consider now the finite amplification of $\wt{G}$ given by $\wt{\Sigma} = (\wt{G}, \widecheck{m}_{\wt{G}}) \times X$. We endow $\wt{\Sigma}$ with a measure-preserving $F_r$-action that commutes with the $\wt{G}$-action by setting $w (\wt{g}, x) = (\wt{g} \wt{\alpha}(w, x)\inv, wx)$, observing that the given $F_r$-action on $X$ preserves the measure by Remark~\ref{rem:cocycle}.

 We now need to prove that $\wt{\Sigma}$ is an amplification of $F_r$, i.e.\ that there is a fundamental domain in $\wt{\Sigma}$ for the countable group $F_r$. For this, consider the measurable space isomorphism $N \times G \times X \simeq \wt{G} \times X$ given by $(n, g, x) \mapsto (\tau(g)n, x)$; when each group is endowed with a right Haar measure, this isomorphism preserves the measure, up to a scaling factor, thanks to Weil's integration formula (see Proposition~10 in~\cite[VII \S 2 \No 7]{NBourbakiINT78}). By transferring the $F_r$-action via this isomorphism, the action on the former space is given by
$$w (n, g, x) = (n', g \alpha(w, x)\inv, wx) \kern5mm \text{for $w \in F_r$, $n \in N$, $g \in G$, $x \in X$,}$$
where the specific expression $n'=\tau(g\alpha(w, x)\inv)\inv \tau(g) n \wt{\alpha}(w, x)\inv$ is irrelevant for our purpose. Indeed, we only need to observe that the above $F_r$-action on $\wt{\Sigma}$ is a twisted product with $N$ of the given action on $\Sigma$, namely $w (g, x) = (g \alpha(w, x)\inv, wx)$. It therefore admits $N \times \sF$ as a fundamental domain.
\end{proof}

\section{A generalization of the Gaboriau--Lyons theorem}\label{sec:GL}
In this section, we establish a generalization of the main result of~\cite{Gaboriau-Lyons} to certain totally disconnected groups. We then deduce the existence of tychomorphisms from free groups to these groups, Theorem~\ref{thm:tychoGL} below. The full generality of Theorem~\ref{thm:tycho} will be established in Section~\ref{sec:main}.

\medskip
Recall that the \emph{core} of a subgroup $K<G$ is the normal subgroup $\Core_G(K)=\bigcap_{g\in G} K^g$ of $G$; thus it is the kernel of the $G$-action on $G/K$.

\begin{thm}\label{thm:tychoGL}
Let $G$ be a non-amenable unimodular compactly generated locally compact second countable group and $K<G$ a compact open subgroup.

Then there exists a tychomorphism from $F_2$ to $G/\Core_G(K)$.
\end{thm}

When $G$ is discrete, the assumptions are simply that $G$ is finitely generated and non-amenable (one takes $K$ trivial). In that case, the Gaboriau--Lyons theorem states that a suitable Bernoulli shift of $G$ contains the orbits of a free $F_2$-action. This implies that there is a ``randembedding'' from $F_2$ to $G$ (see~\cite[\S~5]{MonodICM} or ~\cite{Monod-Ozawa2009}). We shall follow the strategy of Gaboriau--Lyons closely until changes are imposed by the non-discreteness.

\medskip

One difficulty mentioned in the introduction is that we need a stronger conclusion since simply having $F_2$-orbits within the orbits of a locally compact group $G$ does not correspond to any form of non-amenability unless some discreteness condition is imposed.

\begin{rem*}
  Two proofs are proposed in~\cite{Gaboriau-Lyons}, each needing different adjustments to be generalized; therefore, we shall give all the details of the approach taken below. Another exposition of~\cite{Gaboriau-Lyons} is given in~\cite{Houdayer_bourbaki}; there, the need of ergodicity for Hjorth's theorem in~\cite{Hjorth_attained} relies on indistinguishable clusters for the free minimal spanning forest (Conjecture~6.11 in~\cite{Lyons-Peres-Schramm06}), which is bypassed using a result~\cite{Chifan-Ioana} not established in the non-discrete setting. After circulating a first version of this article, we were informed by Itai Benjamini that indistinguishability has just been established in two preprints~\cite{Hutchcroft-Nachmias_arx, Timar_forest_arx}. In conclusion, a second approach becomes possible just as in~\cite{Gaboriau-Lyons} and~\cite{Houdayer_bourbaki}.
\end{rem*}

\begin{proof}[Proof of Theorem~\ref{thm:tychoGL}]
All the assumptions are preserved if we replace $G$ and $K$ by their images in $G/\Core_G(K)$; therefore we can simply assume that $K$ has trivial core in $G$. We can chose a compact generating set $S\se G$ with $S=S\inv$, $e\in S$ and $S=KSK$. Let $n$ be a positive integer to be chosen shortly and consider the graph $\gr:=\gr^n(G,K,S)$ defined in Section~\ref{sec:lengths-measures}. Then $G$ is a vertex-transitive closed subgroup of the automorphism group of $\gr$ and in particular $\gr$ is, by definition, a unimodular graph. By construction, the spectral radius of $\gr$ is $\ro^n$ for some $0<\ro\leq 1$ which is the spectral radius of $\gr(G,K,S)$. We recall here that $\ro<1$ because $G$ is non-amenable, see~\cite[Thm.~1(c)]{Soardi-Woess}. Therefore, we can and do choose $n$ large enough so that the spectral radius of $\gr$ is~$\ro^n<1/9$. We denote by $E$ the set of edges of $\gr$ and consider the compact metrizable $G$-space $[0,1]^E$. We define $X\se [0,1]^E$ to be the $G$-invariant $\mathrm{G}_\delta$-subset of injective maps $E\to[0,1]$ and observe that $G$ acts freely on $X$ since it acts faithfully on $E$ by triviality of the core. Let $\sQ$ be the corresponding orbit equivalence relation on $X$ defined by $x\sQ x'$ iff $x'\in G x$.

Since $K$ is compact, its action on $X$ has a Borel fundamental domain $Y\se X$ (see e.g.~\cite[5.4.3]{Srivastava}). Thus, the action map defines a Borel isomorphism $K\times Y \cong KY=X$. We denote by $\sR$ the equivalence relation on $Y$ defined by restricting $\sQ$, i.e.\ $y\sR y'$ iff $y'\in Gy$. Then $\sQ$ decomposes along $K\times Y \cong X$ as the product $\sQ=\sT_K\times \sR$ of the fully transitive relation $\sT_K$ on $K$ with $\sR$ on $Y$. In particular $\sR$ is a Borel equivalence relation with countable classes. We obtain a \emph{graphing} of $\sR$ (in the sense of~\cite{Gaboriau00,Levitt95}) by transporting to the $\sR$-orbit $[y]_\sR$ of $y\in Y$ the graph $\gr$. More explicitly, we have a bijection $G/K\to [y]_\sR$ mapping a vertex $gK$ of $\gr$ to the unique element of the set $Kg\inv y \cap Y$. Notice that this graph structure on $[y]_\sR$ depends indeed only on the $\sR$-class of $y$. By construction, each $\sR$-class is then isomorphic as a graph to the connected graph $\gr$.

For a given parameter $0<p<1$, consider the map from $X$ to the space of subgraphs of $\gr$ given as follows: an edge $a\in E$ is kept at $x\in X$ iff $x(a)\leq p$. This provides us with the \emph{cluster equivalence subrelation} $\sQ\cl\se\sQ$ on $X$ defined as in~\cite{Gaboriau05} by declaring $x\sQ\cl x'$ iff $x'=g\inv x$ and the subgraph $x\leq p$ connects $gK$ to $eK$. Let $\sR\cl \se \sR$ be the restriction of $\sQ\cl$ to $Y$; one checks again that $\sQ\cl$ is the product $\sT_K\times \sR\cl$.

We now endow $X$ with the (restriction of the) product $\teta$ of Lebesgue measures on $[0,1]^E$. Then $\teta$ is preserved and ergodic under $G$; ergodicity is very classical and can be proved e.g.\ by the same argument as in~\cite[2.1]{Kechris-Tsankov}. We consider the maps described above by $x\leq p$ as a random variable on $(X,\teta)$ with values in the space of subgraphs of $\gr$; this is a $G$-invariant $p$-Bernoulli bond percolation on $\gr$ with scenery in the sense of~\cite[3.4]{Lyons-Schramm}. Therefore, we can apply the results of~\cite{Lyons-Schramm} stating that this percolation process has indistinguishable infinite clusters in the sense of~\cite[3.1]{Lyons-Schramm}.

At this point we record the fact that there exist choices of the parameter $p$ such that the corresponding random subgraphs of $\gr$ have $\teta$-almost surely infinitely many infinite clusters, each of which having uncountably many ends. Indeed, in view of Theorems~1.2 and~6.1 in~\cite{Haggstrom-Peres}, it suffices to show that the critical probability $p_\mathrm{c}$ for $\gr$ is strictly below the critical uniqueness probability $p_\mathrm{u}$ (see also~\cite[3.10]{Lyons-Schramm} for a proof of~\cite[6.1]{Haggstrom-Peres}). The latter condition follows if the free, respectively wired minimal forests are distinct processes on $\gr$, see~\cite[3.6]{Lyons-Peres-Schramm06}. That property holds for our choice of $\gr$; indeed, the free minimal forest has expected degree~$>2$ by~\cite[Thm.~1]{Thom_degree} and the fact that the spectral radius of $\gr$ is~$<1/9$; on the other hand the wired minimal forest has degree~$2$ by~\cite[3.12]{Lyons-Peres-Schramm06}. Alternatively, one can apply the earlier~\cite{Pak-Smirnova-Nagnibeda} instead of~\cite{Thom_degree}.

Let $X_\infty\se X$ be the set of points whose $\sQ\cl$-class is infinite; this is non-null by the above discussion. The proof of Proposition~5 in~\cite{Gaboriau-Lyons} applies in this setting and shows that, by the indistinguishability established above, the restriction of $\sQ\cl$ to $X_\infty$ is ergodic. If we set $Y_\infty = Y\cap X_\infty$, we have $X_\infty = K Y_\infty$. Furthermore, by measure disintegration, there is a unique Borel probability measure $\eta$ on $Y$ such that $\teta$ is the product of $\eta$ by the normalized Haar measure of $K$; moreover, $\sR$ preserves $\eta$. It follows that $Y_\infty$ is non-null for $\eta$ and that the restriction of $\sR\cl$ to  $Y_\infty$ is $\eta$-ergodic in view of the decomposition $\sQ\cl\cong \sT_K\times \sR\cl$. This further shows that $\sR\cl$ has $\eta$-almost surely infinitely many infinite clusters, each of which with uncountably many ends.

At this point we can argue exactly as in Propositions~12, 13 and~14 of~\cite{Gaboriau-Lyons} and apply Hjorth's result~\cite{Hjorth_attained} to deduce that $\sR$ contains a subrelation which is produced by a measure-preserving a.s.\ free (ergodic) action of the free group $F_2$ on two generators upon the space $(Y,\eta)$.

\medskip

As usual in this setting, we endow $\sR$ with the measure given by integrating over $(Y,\eta)$ the counting measure of each equivalence class. We thus obtain a $\sigma$-finite measure on $\sR$ which is preserved by the $F_2$-action on the second coordinate of $\sR\se Y\times Y$ (the action on the first coordinate would work just as well). Moreover, there exists a positive measure fundamental domain $Z\se \sR$ for this $F_2$-action, obtained by choosing for each $y\in Y$ in a Borel way representatives for the $F_2$-classes within the $\sR$-class of $y$. (This procedure is described in detail e.g.\ in~\cite[2.2.2]{Epstein_PhD} in the case where $\sR$ comes from a group action, which is not a restriction~\cite{Feldman-MooreI}.)

We extend the $F_2$-action on $Y$ to a (non-ergodic) $F_2$-action on $X\cong K\times Y$ by letting $F_2$ act trivially on $K$. The resulting relation is contained in $\sQ$ and therefore we have a $F_2$-action on the second coordinate of $\sQ$. This action admits
$$\sT_K\times Z = \big\{(ky, k'y') : k, k'\in K \text{ and } (y,y')\in Z\big\}\ \se\ \sQ$$
as a fundamental domain. We endow $\sQ$ with a $G$-invariant measure $m$ by pushing forward the measure $m_G\times \teta$ under the identification
$$ G\times X \lra \sQ, \kern5mm (g,x)\longmapsto (gx,x),$$
with $m_G$ being a Haar measure on $G$. Then $(\sQ, m)$ is a finite amplification of the $G$-space $(G, m_G)$. Moreover, the $F_2$-action preserves $m$ because $G$ is unimodular; thus, since we have found a non-$m$-null fundamental domain $\sT_K\times Z$ for $F_2$, the $F_2$-space $(\sQ, m)$ is an amplification of $F_2$. Therefore, $(\sQ, m)$ endowed with the commuting $G$-\ and $F_2$-actions is indeed a tychomorphism from $F_2$ to $G$.
\end{proof}

\section{Moderate induction}\label{sec:induction}

Using moderate measures, we propose a new variant of the classical induction of representations and of cocycles from closed subgroups. In fact, we shall consider the more general case of tychomorphisms instead of subgroups only.

\medskip
We first need to clarify the meaning of continuity for representations. Consider thus a linear representation $\ro\colon G\to \GL(V)$ of a topological group $G$ on a Banach space $V$, where $\GL$ denotes the group of invertible continuous linear maps. It is well-known that several possible definitions of continuity coincide for unitary or even uniformly bounded representations (see for instance Lemma~2.4 in~\cite{Bader-Furman-Gelander-Monod}). But for general representations $\ro$ as above, even the definitions are already trickier because neither the weak nor strong operator topologies are compatible with the group structure of $\GL(V)$ (compare e.g.~\cite[I\,\S 3.1]{Dixmier6996}).

\medskip
The following lemma  (which slightly refines~\cite[\S 1.3]{Rosendal13}) shows that we can nonetheless unambiguously talk about a \emph{continuous representation} when $G$ is Baire (e.g.\ locally compact) or first-countable. We call $\ro$ \emph{locally bounded} if the operator norm of $\ro$ is a locally bounded function on $G$. 

\begin{lem}\label{lem:cont}
 Let $G$ be a topological group and $\ro\colon G \to \GL(V)$ be a representation on a Banach space $V$. Consider the following assertions:
\begin{enumerate}[label=(\roman*)]
 \item The action map $G \times V \to V$ is jointly continuous.\label{pt:cont:joint}
 \item The orbit map $G \to V\colon g \to \ro(g)x$ is continuous for every $x \in V$ and $\ro$ is locally bounded.\label{pt:cont:sep-bd}
 \item The orbit map $G \to V\colon g \to \ro(g)x$ is continuous for every $x \in V$.\label{pt:cont:sep}
\end{enumerate}
Then~\ref{pt:cont:joint} and~\ref{pt:cont:sep-bd} are equivalent. Moreover, they are also equivalent to~\ref{pt:cont:sep} if $G$ is Baire or admits countable neighbourhood bases.
\end{lem}

It turns out that local boundedness is not automatic for completely general $G$, even if $V$ is a separable Hilbert space; see Section~\ref{sec:rems}. In particular, orbital and joint continuity do not agree in full generality.

\begin{proof}[Proof of Lemma~\ref{lem:cont}]
The implications \ref{pt:cont:joint}$\Rightarrow$\ref{pt:cont:sep-bd}$\Rightarrow$\ref{pt:cont:sep} hold trivially. Notice moreover that $\ro$ is locally bounded if and only if for every net $g_\alpha \to g$ in $G$, the net $\lVert \ro(g_\alpha) \rVert$ is eventually bounded. 

For \ref{pt:cont:sep-bd}$\Rightarrow$\ref{pt:cont:joint}, let $(g_\alpha, x_\alpha)$ be a net converging to $(g, x)$ in $G\times V$ and consider
\begin{equation*}
 \lVert \ro(g_\alpha) x_\alpha - \ro(g) x \rVert \leq \lVert \ro(g_\alpha) \rVert\,\lVert x_\alpha - x \rVert + \lVert \ro(g_\alpha) x - x \rVert.
\end{equation*}
The first term goes to zero as $\ro$ is locally bounded and the second one also converges to zero by orbital continuity.

\medskip
We thus have to prove that orbital continuity implies local boundedness when $G$ is either Baire or first countable.

In the first case, we write $G$ as the increasing union of the subsets $G_n = \{g\in G : \lVert \ro(g) \rVert \leq n \}$. Since $G_n$ can be written as the intersection over all $x\in V$ of all $g$ satisfying $\lVert \ro(g) x \rVert \leq n\|x\|$, it is a closed set by orbital continuity. The Baire condition then implies that $G_n$ has some interior point $g$ when $n$ is large enough. Now any $h\in G$ admits $G_n g \inv h$ as a neighbourhood on which $\lVert \ro \rVert$ is bounded by $n \lVert \ro(g\inv h) \rVert$.

If on the other hand $G$ is first countable, then local boundedness can be checked on sequences instead of nets; moreover, for sequences, eventual and actual boundedness coincide. We thus show that for every sequence $g_n \to g$ in $G$, the sequence $\lVert \ro(g_n) \rVert$ is bounded. By orbital continuity, $\ro(g_n) x$ converges and hence is bounded for all $x\in V$. The conclusion now follows from the Banach--Steinhaus uniform boundedness principle~\cite[II.3.21]{Dunford-Schwartz_I}.
\end{proof}

\begin{thm}\label{thm:inductiontycho}
Let $G$ and $H$ be locally compact second countable groups with a tychomorphism from $H$ to $G$. Assume that $G$ carries a moderate measure.

If $H$ admits a continuous affine fixed point free action on a separable Hilbert space with a bounded orbit, then so does $G$.
\end{thm}

\begin{proof}
Let $V$ be a separable Hilbert space with a continuous affine action $\alpha$ of $H$ having a bounded orbit. Upon conjugating by a translation, we can assume that the orbit of the origin is bounded, i.e.\ that the cocycle $b\colon G\to V$ of $\alpha$ is bounded. Let $(\Sigma, m)$ be a tychomorphism from $H$ to $G$ realized as a finite amplification $(G, \widecheck{m}_G) \times (X, \teta)$ and denote by $\chi$ an $H$-equivariant retraction $\Sigma \to H$ (i.e.\ $\chi$ is the projection onto $H$, when $\Sigma$ is seen as an amplification of $H$).

First, let us twist the measure $m$ by a moderate measure $\mu$ on $G$, i.e.\ define a new measure $\wt{\mu}$ by pushing the product measure on $(G, \mu) \times (X, \teta)$ to $\Sigma$ through the isomorphism given by the tychomorphism. Thus $\wt{\mu}$ is a probability measure on $\Sigma$ in the same class as $m$.

We define an action of $G \times H$ on functions $f\colon \Sigma \to V$ by
\begin{equation*}
 (g f ) (s) = f(g\inv s), \qquad \qquad (h f ) (s) = f(h\inv s).
\end{equation*}
As $G \times H$ preserves the class of the measures $m$ and $\wt\mu$ on $\Sigma$, this action is also well defined on the equivalence classes of measurable functions.

Consider now the Hilbert space $\sH = L^2 (\Sigma, \wt{\mu}; V)$. The fact that $\mu$ is moderate implies that $G$ preserves $\sH$. Let us check that this action is continuous, i.e.\ that the orbit maps $G \to \sH$ are continuous (cf.\ Lemma~\ref{lem:cont}, since local boundedness is readily implied by the definition of a moderate measure). Let thus $g_n\to g$ be a convergent sequence in $G$. Thanks to the local boundedness, it suffices to check the continuity of the orbit maps only for points in some dense subset $E$ of $\sH$. By the $\sigma$-finiteness of $m$, we can take for $E$ the subspace spanned by indicator functions of measurable sets $A \subset \Sigma$ with finite $m$-measure. But for these functions, the $L^2(\wt{\mu})$-convergence of $g_n \mathbf{1}_A$ to $g \mathbf{1}_A$ is equivalent to the $L^2(m)$-convergence, because $m$ and $\wt{\mu}$ are two equivalent $\sigma$-finite measures. Since the modular homomorphism $\Delta_G$ is continuous, this is nothing but a special case of the continuity of $\lambda$, the (isometric) left regular representation of $G$ on $L^2 (\Sigma, m; V)$ given by $(\lambda(g) f)(s) = \Delta_G (g)^\frac{1}{2} f(g\inv s)$. By Proposition~1.1.3 in~\cite{Monod}, in order to get the latter, it is enough to check that orbit maps are norm-measurable, for which we borrow the argument from appendix~D, Theorem~1.2.1 in~\cite{Buehler_AMS}. For any $\xi, \eta \in L^2 (\Sigma, m; V)$, consider the map
\begin{equation*}
 g \mapsto \langle \lambda(g) \xi, \eta \rangle = \Delta_G (g)^{\frac{1}{2}} \int_\Sigma \langle \xi (g\inv s), \eta (s) \rangle_V\, \mathrm{d}m(s).
\end{equation*}
By definition of a tychomorphism, $(g, s) \mapsto \langle \xi (g\inv s), \eta (s) \rangle_V$ is a measurable function on $G \times \Sigma $, which is moreover bounded if $\xi$ and $\eta$ are essentially bounded. Hence an application of~\cite[17.25]{Kechris_book} shows that $g \mapsto \langle \lambda(g) \xi, \eta \rangle$ is measurable when $\xi$ and $\eta$ are bounded, hence for any $\xi$ and $\eta$ by density. In particular, the orbit maps $g \mapsto \lambda(g) \xi$ are weakly measurable. By~\cite[Lemma~3.3.3]{Monod}, they are also norm-measurable, as desired.

Consider now the subset $W \subset \sH$ of functions $f$ that are $H$-equivariant in the sense that $h f = \alpha (h) \circ f$ holds for every $h \in H$. Then $W$ is a $G$-invariant closed subset that is stable under affine combinations. Moreover, for any $g \in G$, the function $\beta_g\colon\Sigma \to V\colon s \mapsto b\chi(g\inv s)$ lies in $W$, has norm bounded by that of $b$ and satisfies $\beta_g = g \beta_e$. Hence the restriction of the $G$-action to $W$ is a continuous affine action of $G$ with a bounded orbit.

 Lastly, suppose for the sake of a contradiction that there is a $G$-fixed point in $W$. This would then descend to a measurable function $X \to V$ which is $H$-equivariant. Moreover, the latter is $\teta$-integrable thanks to the $L^2$-condition since $\teta$ is finite. Therefore, averaging over $X$, we would get a fixed point in $V$ since the $H$-action on $X$ preserves $\teta$ by Remark~\ref{rem:cocycle}.
\end{proof}

\section{Proofs of Theorems~\ref{thm:main} and~\ref{thm:tycho}}\label{sec:main}

Let $G$ be any locally compact group. Recall that $G$ admits a maximal normal amenable closed subgroup, the \emph{amenable radical} $\Ramen(G)$. For Theorem~\ref{thm:main}, we will use the following general splitting result from~\cite{Burger-Monod3} which relies notably on the solution to Hilbert's fifth problem.

\begin{thm}[Thm.~3.3.3 in~\cite{Burger-Monod3}]\label{thm:split}
Let $G$ be any locally compact group.

The quotient group $G/\Ramen(G)$ has a finite index open characteristic subgroup which splits as a direct product $S\times D$ where $S$ is a connected semi-simple Lie group with trivial center and no compact factors and $D$ is totally disconnected.

Moreover, $D$ is second countable if $G$ is $\sigma$-compact.\qed
\end{thm}

The statement on second countability is not made explicitly in loc.\ cit., but it follows from Satz~5 in~\cite{Kakutani-Kodaira} that when $G$ is $\sigma$-compact, it has a second countable quotient by a compact normal subgroup --- which is necessarily contained in $\Ramen(G)$.

\bigskip

We first prove Theorem~\ref{thm:tycho}. Let thus $G$ be a non-amenable locally compact second countable group. Since there are inclusions $F_r < F_2$ for all $0\leq r\leq \aleph_0$ we only consider $r=2$, using Lemma~\ref{lem:comp}.

We shall distinguish two cases; as a first case, suppose that the connected component $G^\circ$ of the identity is already non-amenable. Then $G$ contains a discrete non-abelian free subgroup:

\begin{lem}\label{lem:ping-pong}
A connected locally compact group is amenable if and only if is does not contain any discrete non-abelian free subgroup.
\end{lem}

\begin{proof}
Let $H$ be a non-amenable connected locally compact group and let $L=H/\Ramen(H)$ be the quotient by its amenable radical. In the present case, there is no splitting and Theorem~\ref{thm:split} is a direct consequence of the solution to Hilbert's fifth problem; it implies that $L$ is a semi-simple Lie group of positive $\RR$-rank. In particular, it contains a closed subgroup $L_1<L$ of rank one. The classical ping-pong lemma applied to suitable hyperbolic elements of $L_1$ acting on the boundary of the symmetric space of $L_1$ provides a free subgroup $F_2$ of $L_1$ (modulo its center) and shows moreover that this $F_2$ is discrete. Being free, it can be lifted to $H$ and any such lift remains discrete.

The converse is elementary since \emph{closed} subgroups of amenable locally compact groups remain amenable (see~\cite[Theorem~3.9]{Rickert67}).
\end{proof}

Therefore we have a tychomorphism from our discrete non-abelian free subgroup to $G$ as indicated in Example~\ref{ex:subgroup} and thus also one from $F_2$ to $G$ by Lemma~\ref{lem:comp}.

\medskip

The second case is when $G^\circ$ is amenable. By Proposition~\ref{prop:liftfree}, it suffices to produce a tychomorphism from $F_2$ to $G / G^\circ$. Let $G_1<G/ G^\circ$ be the kernel of the modular homomorphism of $G / G^\circ$. This closed normal subgroup is unimodular; in fact, it is the maximal unimodular closed normal subgroup, see Proposition~10 in~\cite[VII \S 2 \No 7]{NBourbakiINT78}. Observe that $G_1$ also remains totally disconnected and non-amenable. By considering the directed family of subgroups of $G_1$ generated by a compact neighbourhood of the identity, we can choose a compactly generated subgroup $G_2 < G_1$ which is non-amenable, and still unimodular since it is open. Being totally disconnected, it contains some compact-open subgroup $K$. By Theorem~\ref{thm:tychoGL}, there is a tychomorphism from $F_2$ to the quotient of $G_2$ by the core of $K$, hence there exists one from $F_2$ to $G_2$ by Proposition~\ref{prop:liftfree}. We can compose it with the inclusions $G_2 < G_1 < G / G^\circ$ thanks to Lemma~\ref{lem:comp} in order to get a tychomorphism from $F_2$ to $G / G^\circ$, as desired.

This finishes the proof of Theorem~\ref{thm:tycho}.

\medskip
For the Lie oriented reader, we point out the following reformulation of Theorem~\ref{thm:tycho} (and its easy converse).

\begin{cor}
Let $G$ be a locally compact second countable group. Then $G$ is amenable if and only if it does not admit a tychomorphism from $\SL_2(\RR)$.
\end{cor}

Indeed the corollary follows by composition of tychomorphisms since any locally compact group is measure equivalent to its lattices, which in the case of $\SL_2(\RR)$ include non-abelian free groups.

\bigskip

We now turn to Theorem~\ref{thm:main}. The starting point is provided by Section~\ref{sec:initial}, which shows that for instance the free group $F_4$ has a fixed point free affine action on a Hilbert space with a bounded orbit. The same holds for any $F_r$ with $2\leq r\leq \aleph_0$ by Theorem~\ref{thm:inductiontycho} since $F_4<F_r$. We will not apply Theorem~\ref{thm:tycho} in its full generality because we did not construct moderate measures for all groups (though this can be done using~\cite{Haagerup-Przybyszewska}, see Section~\ref{sec:lengths-measures}). Anyway, it suffices to prove Theorem~\ref{thm:main} for the quotient $G/\Ramen(G)$, which admits a moderate measure in view of Theorem~\ref{thm:split}, Lemma~\ref{lem:measindfin} and Corollary~\ref{cor:ex:moderate} (remark that if $G_1$ and $G_2$ are two groups carrying moderate lengths $\ell_i$, then $\ell (g_1, g_2) := \ell_1 (g_1) + \ell_2 (g_2)$ is a moderate length on their product). Hence Theorem~\ref{thm:inductiontycho} and Theorem~\ref{thm:tycho} provide a fixed point free affine action on a Hilbert space with a bounded orbit, finishing the proof of Theorem~\ref{thm:main}.

\section{Proof of Theorem~\ref{thm:wr}}\label{sec:wr}

We begin by recording an elementary common knowledge fact.

\begin{lem}\label{lem:open}
Let $H$ be an open subgroup of a topological group $G$. If $H$ admits a non-unitarisable uniformly bounded continuous representation on a Hilbert space, then so does $G$.
\end{lem}

\begin{proof}
Given the $H$-representation on $V$, we can form the \emph{usual} induced $G$-representation on $W=\ell^2(G/H, V)$, where $G/H$ is endowed with the counting measure and $G$ acts on $f\in W$ by $(gf)(x) = \gamma(g\inv, x)\inv f(g\inv x)$, where $\gamma\colon G\times G/H \to H$ is the (continuous) cocycle determined by a choice of representatives for the cosets. One verifies that this is indeed a well-defined uniformly bounded continuous representation. Since $H$ is open, the $H$-representation $V$ is contained in the restriction of $W$ from $G$ to $H$ as the space of maps supported on the trivial coset, whence the conclusion.
\end{proof}

We can and shall assume that $G$ is $\sigma$-compact and $A$ countable (then $A\wr_{G/O} G$ is also $\sigma$-compact). For let $A_0<A$ be an infinite countable subgroup and $G_0<G$ an open $\sigma$-compact subgroup (e.g.\ the subgroup generated by a compact neighbourhood of the identity). If $O<G_0$ is an open subgroup such that $A_0\wr_{G_0/O} G_0$ is non-unitarisable, then indeed $A\wr_{G/O} G$ is also non-unitarisable by Lemma~\ref{lem:open} since it contains $A_0\wr_{G_0/O} G_0$ as an open subgroup (noting that $O$ remains open in $G$).

We distinguish cases as in Section~\ref{sec:main}; suppose first that the connected component $G^\circ$ is amenable.

Then $G_1=G/\Ramen(G)$ is totally disconnected, non-amenable and second countable; moreover, the core of any compact-open subgroup $K<G_1$ is trivial. We choose such a $K$ and define $O<G$ to be its pre-image in $G$. Notice that it suffices to show that the group $A\wr_{G_1/K} G_1$ is non-unitarisable since it is a quotient of $A\wr_{G/O} G$. By Theorem~\ref{thm:tycho}, there exists a tychomorphism from $F_2$ to $G_1$. Let 
$$(\Sigma, m) \cong  (G_1, \widecheck{m}_{G_1})\times (X,\teta)  \cong  (H, m_H)\times (Z,\eta)$$
be the corresponding $G_1\times H$-space and its realizations as amplifications with associated cocycles $\alpha$ and $\beta$. At this point, we can for a while follow exactly the path taken in~\cite{Monod-Ozawa2009}:

Recall first that the non-unitarisability of $F_2$ is implied by the following a priori stronger statement: there is a unitary representation $(\pi, V)$ of $F_2$ such that the bounded cohomology $\mathrm{H}^1_\mathrm{b}(F_2, \sL(V))$ is non-trivial, where $\sL(V)$ is the Banach space of bounded operators of $V$ endowed with the isometric $F_2$-module structure given by conjugation by $\pi$. We refer to Lemma~4.5 in~\cite{PisierLNM} for this fact, recalling that $\pi$ can be taken to be the regular representation on $V=\ell^2(F_2)$ and that a (straightforward) translation from derivations to $1$-cocycles has to be made. We now consider the (classical unitarily) induced unitary $G_1$-representation $\sigma$ on $W=L^2(Z, \eta; V)$ defined by
$$ (\sigma(g) f)(s) = \Delta_{G_1}(g)^{-\frac12}\pi\big(\beta(g\inv, s)\inv\big) f(g\inv s).$$
We have again a corresponding $G_1$-representation on the von Neumann algebra $\sL(W)$ given by conjugation by $\sigma$. This representation preserves the subalgebra $E=L^\infty_\mathrm{w*}(Z,\sL(V))$ of weak-* measurable bounded function classes. In fact, the resulting $G_1$-representation on $E$ is none other than the ``$L^\infty$-induced'' representation (i.e\ without the $\Delta_{G_1}$ factor) associated to the above $F_2$-representation on $\sL(V)$. In relation to this $L^\infty$-induced representation there is also a cohomological induction map
$$\mathrm{H}^1_\mathrm{b}(F_2, \sL(V)) \lra \mathrm{H}^1_\mathrm{b}(G_1, E).$$
The latter is injective. Indeed, the proof given in~\cite[\S 4.3]{Monod-Shalom2} for the case of ME couplings of discrete groups holds without changes.

\smallskip
In addition to its $G_1$-structure, $\sL(W)$ is a module over the von Neumann algebra $L^\infty(Z)$ and this module structure is compatible with the $G_1$-representation on $L^\infty(Z)$ in the sense that conjugating by $g\in G_1$ the action of some $\fhi\in L^\infty(Z)$ on $\sL(W)$ simply gives the action of $g\fhi= \fhi\circ g\inv$. It follows that we will turn $\sL(W)$ into a (dual isometric) coefficient $A\wr_{G_1/K} G_1$-module as soon as we define any representation of $A$ into the unitary group of the subalgebra $L^\infty(Z)^K$ of $L^\infty(Z)$. Notice furthermore that the $L^\infty(Z)$-module structure of $\sL(W)$ preserves $E$ so that $E$ will also inherit that $A\wr_{G_1/K} G_1$-action.

\begin{rem}\label{rem:smooth}
We will need repeatedly that for any compact subgroup $K'<G_1$ the algebra $L^\infty(Z)^{K'}$ is canonically identified with $L^\infty(K'\backslash Z)$, where $K'\backslash Z$ is an ordinary quotient of $Z$ (in contrast to the case of non-compact groups, where one needs to introduce a space of ergodic components which would be unsuitable for some arguments below). This is the case because the compactness of $K'$ ensures that its action on $Z$ is smooth in the ergodic-theoretical sense, as follows from a result of Varadarajan (Theorem~3.2 in~\cite{Varadarajan63}).
\end{rem}

We now choose an $A$-representation into the unitary group of $L^\infty(K\backslash Z)$ that generates $L^\infty(K\backslash Z)$ as a von Neumann algebra (this exists e.g.\ by the argument given on page~257 in~\cite{Monod-Ozawa2009}). We write $N=\bigoplus_{G_1/K} A$, so that $A\wr_{G_1/K} G_1=N\rtimes G_1$.

\begin{lem}
Under the embedding $E\to \sL(W)$, we have ${E =  \sL(W)^N}$.
\end{lem}

\begin{proof}
In view of Theorem~IV.5.9 in~\cite{TakesakiI}, what we need to prove is the following claim. The von Neumann subalgebra of $L^\infty(Z)$ generated by the union of all $g L^\infty(Z)^{K} = L^\infty(Z)^{g Kg\inv}$, where $gK$ ranges over $G_1/K$, is $L^\infty(Z)$ itself.

We first record that for any two compact subgroups $K', K''<G_1$ the von Neumann algebra generated by $L^\infty(Z)^{K'}$ and $L^\infty(Z)^{K''}$ coincides with $L^\infty(Z)^{K'\cap K''}$ (this can fail for non-compact groups). Indeed this follows immediately from the consequence of Varadarajan's theorem indicated in Remark~\ref{rem:smooth}. Consider now the net $L^\infty(Z)^{K'}$ indexed by the directed set of all finite intersections $K'<G_1$ of $G_1$-conjugates of $K$. The normalized integration over $K'$ provides a conditional expectation from $L^\infty(Z)$ to $L^\infty(Z)^{K'}$ for each $K'$ which turns this net into an inverse system. Therefore, the martingale convergence theorem implies that the algebra generated by the union of all $L^\infty(Z)^{K'}$ is weak-* dense in the von Neumann subalgebra $B<L^\infty(Z)$ of functions that are measurable with respect to the common refinement of all partitions defined by all $Z\twoheadrightarrow K'\backslash Z$. (Although general nets bring complications to martingale convergence~\cite{Krickeberg}, here we can anyway restrict to a cofinal \emph{sequence} since $G_1$ is second countable.) Since $K$ has trivial core in $G_1$, it follows (using again Varadarajan's theorem) that in fact the common refinement is trivial and hence $B=L^\infty(Z)$. This implies the claim and thus finishes the proof.
\end{proof}

We can now conclude the proof exactly as in~\cite{Monod-Ozawa2009}: We know that $\mathrm{H}^1_\mathrm{b}(G_1, E)$ is non-trivial and that ${E\cong  \sL(W)^N}$. Moreover, the amenability of $N$ implies that $\mathrm{H}^1_\mathrm{b}(G_1, \sL(W)^N)$ can be identified with $\mathrm{H}^1_\mathrm{b}(N\rtimes G_1, \sL(W))$, see~\cite[7.5.10]{Monod}. Appealing again to Lemma~4.5 in~\cite{PisierLNM}, we deduce that the group $N\rtimes G_1$ is non-unitarisable, finishing the proof of Theorem~\ref{thm:wr} in the case where $G^\circ$ is amenable.

\medskip
If on the other hand $G^\circ$ is non-amenable, then Theorem~\ref{thm:split} implies that after passing to an open subgroup (which we can by Lemma~\ref{lem:open}), our group has a quotient which is a connected non-compact simple Lie group. As recalled in the Introduction, the particular example of $\SL_2(\RR)$ was actually the first example of a non-unitarisable group. Thanks to the substantial theory available for representations of simple Lie groups, this example is known to generalize to all connected non-compact simple Lie groups, see Remark~0.8 in~\cite{Pisier_survey}. In conclusion, $G$ itself is non-unitarisable in this case, and hence so is a fortiori any extension of $G$. This completes the proof of Theorem~\ref{thm:wr}.

\section{Remarks and questions}\label{sec:rems}

Our proof of Theorem~\ref{thm:main} is rather indirect and uses tychomorphisms.

\begin{question}
Is there an elementary proof of Theorem~\ref{thm:main}, not relying on Theorem~\ref{thm:tycho}?
\end{question}

This question arises even for discrete groups.

\begin{rem}
 The classical fixed point property for amenable groups concerns compact convex sets in general locally convex spaces. Consider the following weakened property for a topological group $G$:

\itshape
Whenever $G$ acts continuously on a locally convex space $V$ by \emph{weakly} continuous affine transformations and preserves some \emph{weakly} compact convex nonempty set $K$ in $V$, it has a fixed point.

\upshape
This property is actually equivalent to the (formally stronger) usual one. Indeed, it is enough to check that continuous actions by weak-* continuous affine transformations on a dual Banach space $V$ preserving some weak-* compact convex nonempty set are encompassed in this situation (see~\cite[Remarks~4.3]{Rickert67}). But the weak topology of $V$ endowed with its weak-* topology is none other than its weak-* topology, whence the claim. 
\end{rem}

\begin{rem}
The action built in Section~\ref{sec:initial} has the additional property that the orbit of $0$ is not only bounded but relatively compact in norm. Indeed, the inclusion operator $\ell^\infty(G) \to \ell^2(G,\mu)$ is compact because $\mu$ is finite and atomic. A similar argument applies for moderate induction to discrete groups containing $F_2$. However, a priori this does not hold for induction through tychomorphisms since the corresponding measure space is in general not atomic. This motivates the following: 

\begin{question}
Let $G$ be a locally compact $\sigma$-compact group. Suppose that every continuous affine $G$-action on a Hilbert space preserving a \emph{norm-}compact non-empty set has a fixed point. Does it follow that $G$ is amenable?
\end{question}

Again, the question seems open even for $G$ discrete.

(One can equivalently consider norm-compact \emph{convex} sets in view of Mazur's theorem~\cite{Mazur30}.)
\end{rem}

\begin{rem}
 We already observed in the Introduction that the Banach--Steinhaus uniform boundedness principle forces the considered actions to have at least some unbounded orbit, and actually the points with unbounded orbits are dense. Moreover, the action built in Section~\ref{sec:initial} also has a dense subset of points with \emph{bounded} orbits: the image in $\ell^2 (G, \mu) / \RR$ of the functions with finite support.
\end{rem}

\begin{rem}
 An action on a compact convex subset of a Hilbert space does not necessarily extend to the whole ambient space. Let for instance $V$ be the space $\ell^2 (\ZZ \setminus\! \{0\})$ and $K$ in $V$ be the compact convex subset of points $x$ such that $\left| x_n \right| \leqslant n\inv$ and $\left| x_{-n} \right| \leqslant n^{-2}$ for $n > 0$. Define the linear transformation $T\colon K \to K$ by
\begin{equation*}
 \left(T(x)\right)_n = n x_{-n} \qquad \left(T(x)\right)_{-n} = \frac{x_n}{n}
\end{equation*}
for any $n > 0$. This is a continuous involution of $K$. However, any linear extension of $T$ to $V$ should map $\delta_n$ to $n\delta_{-n}$ for $n > 0$, which is impossible.
\end{rem}

\begin{rem}
In order to give some context to Lemma~\ref{lem:cont}, we propose here a whole family of representations $\ro\colon G\to \GL(V)$ of a topological group $G$ on a Banach space $V$.

Consider any normed vector space $U$. Let $G$ be the additive group of $U$ endowed with the weak topology and let $V=\RR\oplus U^*$, where $U^*$ is the (strong) dual of $U$ and the sum is endowed with the $\ell^2$-sum of the two norms. In particular, $V$ is a separable Hilbert space if $U$ is so.

Finally, the representation $\ro$ is defined by $\ro(u)(t,x) = (t+x(u), x)$. Then every orbital map is continuous by definition of the weak topology. On the other hand, we have $\|\ro(u)\|\geq \|u\|$ and hence $\ro$ is locally bounded if and \emph{only if} $U$ is finite-dimensional.

We note in passing that the proof of Lemma~\ref{lem:cont} used the Banach--Steinhaus uniform boundedness principle; hence the latter fails for general nets, illustrating the specific assumptions made e.g.\ in~\cite[III \S 3 \No 6]{NBourbakiEVTIII}.
\end{rem}

Turning to the Dixmier problem, our next two questions would of course be moot if every (locally compact) unitarisable group were amenable. This seems questionable even in the discrete case; our questions, however, become obvious for discrete groups (see Lemma~\ref{lem:open}).

\begin{question}
Does unitarisability pass to closed subgroups of locally compact groups?
\end{question}

A preliminary question being:

\begin{question}
Let $G$ be a locally compact group containing a discrete (non-commutative) free subgroup. Is $G$ non-unitarisable?
\end{question}



\bibliographystyle{../BIB/amsalpha}
\bibliography{../BIB/ma_bib}
\end{document}